\PassOptionsToPackage{dvipsnames}{xcolor}
\documentclass{aic}

\usepackage{amsthm,amsfonts,amssymb,amsmath}
\usepackage{graphicx, tikz}

\newtheorem{theorem}{Theorem}[section]
\newtheorem{proposition}[theorem]{Proposition}
\newtheorem{lemma}[theorem]{Lemma}

\newtheorem{definition}[theorem]{Definition}

\theoremstyle{definition}

\numberwithin{equation}{section}
\numberwithin{figure}{section}
\allowdisplaybreaks

\usepackage{babel}
\usepackage[T1]{fontenc}
\usepackage[utf8]{inputenc}

\newcommand{\su}{\subseteq}
\newcommand{\sm}{\setminus}
\newcommand{\Z}{\mathbb{Z}}
\newcommand{\HH}{\mathcal{H}}
\newcommand{\conv}{\operatorname{conv}}
\newcommand{\PP}{\mathcal{P}}

\aicAUTHORdetails{%
  title = {A Sharp Ramsey Theorem for Ordered Hypergraph Matchings}, 
  author = {Lisa Sauermann and Dmitrii Zakharov},
  plaintextauthor = {Lisa Sauermann and Dmitrii Zakharov},
  runningtitle = {Ordered Hypergraph Matchings}, 
  keywords = {ordered hypergraph matching, Ramsey theory, Erdos-Szekeres theorem},
}   

\aicEDITORdetails{%
   year={2025},
   number={6},
   received={4 October 2023},   
   published={27 June 2025},  
   doi={10.19086/aic.2025.6},  
}   

\begin{document}

\begin{frontmatter}[classification=text]

\title{A Sharp Ramsey Theorem for Ordered Hypergraph Matchings} 

\author[pgom]{Lisa Sauermann\thanks{Institute for Applied Mathematics, University of Bonn, Bonn, Germany. Email: \href{sauermann@iam.uni-bonn.de}{\nolinkurl{sauermann@iam.uni-bonn.de}}. Research supported by NSF Award DMS-2100157 and a Sloan Research Fellowship.}}
\author[johan]{Dmitrii Zakharov\thanks{Department of Mathematics, Massachusetts Institute of Technology, Cambridge, MA.  Research supported by a Jane Street Graduate Research Fellowship. 
Email: \href{zakhdm@mit.edu}{\nolinkurl{zakhdm@mit.edu}}.}}

\begin{abstract}
We prove essentially sharp bounds for Ramsey numbers of ordered hypergraph matchings, introduced recently by Dudek, Grytczuk, and Ruci\'{n}ski. Namely, for any $r \ge 2$ and $n \ge 2$, we show that any collection $\HH$ of $n$ pairwise disjoint subsets in $\Z$ of size $r$ contains a subcollection of size $\lfloor n^{1/(2^r-1)}/2\rfloor$ in which every pair of sets are in the same relative position with respect to the linear ordering on $\Z$. This improves previous bounds of Dudek--Grytczuk--Ruci\'nski and of Anastos--Jin--Kwan--Sudakov and is sharp up to a factor of $2$. For large $r$, we even obtain such a subcollection of size $\lfloor (1-o(1))\cdot n^{1/(2^r-1)}\rfloor$, which is asymptotically tight (here, the $o(1)$-term tends to zero as $r \to \infty$, regardless of the value of $n$).
    
Furthermore, we prove a multiparameter extension of this result where one wants to find a clique of prescribed size $m_P$ for each relative position pattern $P$. Our bound is sharp for all choices of parameters $m_P$, up to a constant factor depending on $r$ only. This answers questions of Anastos--Jin--Kwan--Sudakov and of Dudek--Grytczuk--Ruci\'nski.
\end{abstract}
\end{frontmatter}


\section{Introduction}

Ramsey Theory is one of the most active areas of combinatorics, with a rich history spanning an entire century and many exciting recent results (e.g.~\cite{Campos, Lee, MV, RR}, to mention just a few). The basic philosophy behind the results in this area is that any sufficiently large structure, even if completely disordered, must contain a homogeneous-looking substructure of a prescribed size. Such Ramsey results have been established in a variety of contexts, and improving the quantitative bounds in these results is of great interest (see the books \cite{graham-rothschild-spencer} and \cite{promel} for a general overview of the area of Ramsey theory).

Here, we are interested in a Ramsey problem about ordered hypergraph matching. An $r$-uniform hypergraph matching $\HH$ is a finite collection of pairwise disjoint sets, each of size $r$. These sets are called hyperedges, and their union is called the vertex set of the hypergraph. An \emph{ordered} $r$-uniform hypergraph matching $\HH$ is such a $r$-uniform hypergraph matching equipped with a total ordering on the vertex set.

Now, given an ordered hypergraph matching $\HH$, we want to find a large submatching in $\HH$ which looks homogeneous with respect to the total ordering on the vertex set.
More precisely, if $\HH$ is an ordered $r$-uniform hypergraph matching, then for any two edges $e,f\in\HH$ we can consider the relative order of the vertices in $e$ and $f$. In other words, we can consider the ordered $r$-uniform hypergraph matching $\{ e,f\}$. This can be represented by a string of length $2r$ consisting of $r$ times the letter $\rm A$ and $r$ times the letter $\rm B$. For example if $e=\{ 1,2,5\}$ and $f=\{ 3,4,6\}$ and the vertices in $e\cup f=\{1,2,\dots,6\}$ are ordered in the natural way, we obtain the pattern $\rm AABBAB$ encoding the relative order of the vertices $e$ and $f$. Note that this pattern is equivalent to the pattern $\rm BBAABA$ upon interchanging the roles of $e$ and $f$. To make the representation unique, let us impose the condition that the string needs to start with the letter $\rm A$. We can now define an \emph{$r$-pattern} to be a string in $\{{\rm A},{\rm B}\}^{2r}$ starting with $\rm A$ and consisting of $r$ times the letter $\rm A$ and $r$ times the letter $\rm B$.

For any ordered $r$-uniform hypergraph matching $\HH$, for any two edges $e,f\in\HH$ we can consider the corresponding $r$-pattern described by $\{ e,f\}$. If there is some $r$-pattern $P$ such that for every pair of edges $e,f\in\HH$ the $r$-pattern described by $\{ e,f\}$ is $P$, then we call the $r$-uniform hypergraph matching a $P$-clique. We say that $\HH$ contains a $P$-clique of size $m$, if there is some subset $\HH'\su \HH$ of $|\HH'|=m$ edges forming a $P$-clique.

Now we have a natural Ramsey-type question: given a hypergraph matching $\HH$ what is the size of the largest $P$-clique in $\HH$ over all choices of the pattern $P$? In the graph case $r=2$, this question was introduced and solved by
Dudek, Grytczuk, and Ruci\'nski~\cite{DGR22}. They showed that any matching of size $m^3+1$ contains a $P$-clique of size $m+1$, where $P$ is one of the three possible 2-patterns: AABB, ABAB, or ABBA. Moreover, they constructed a matching of size $m^3$ without $P$-cliques of size $m+1$ for any pattern $P$. 
In a later work, the same authors~\cite{DGR} systematically studied the higher uniformity version of the problem and showed that any ordered $r$-uniform hypergraph with $C_r m^{3^{r-1}}$ edges contains a $P$-clique of size $m+1$ for some $r$-pattern $P$. On the other hand, they could only construct a matching of size $m^{2^{r-1}+2}$ without any $P$-cliques of size $m+1$. These bounds were substantially improved by Anastos, Jin, Kwan, and Sudakov~\cite{AJKS}: they reduced the upper bound to $C_r m^{(r+1) 2^{r-2}}$ and constructed an example with $m^{2^{r}-1}$ edges, thus, showing that the correct exponent of $m$ is roughly $2^r$ rather than $3^r$. They also got a sharp exponent of $m$ in the cases $r=3$ and $r=4$. However, their techniques required a substantial case analysis already for $r=4$ and do not seem to generalize to large $r$.

In this paper, we determine the correct behavior of the Ramsey problem for ordered $r$-uniform hypergraph matchings for all values of $r$.

\begin{theorem} \label{thm-main}
Let $r$ and $m$ be positive integers. Let $\HH$ be an ordered $r$-uniform hypergraph matching. For every $r$-pattern $P$, assume that every $P$-clique contained in $\HH$ has size at most $m$. Then the number of edges in $\HH$ is bounded by
\[|\HH|\le C_r\cdot m^{2^r-1},\]
where $C_r=2^{r-1}(r-1)!$.
\end{theorem}

By the aforementioned construction of Anastos, Jin, Kwan, and Sudakov~\cite[Theorem 1.7]{AJKS}, the upper bound in Theorem~\ref{thm-main} is tight up to the constant factor $C_r$. Theorem~\ref{thm-main} can also be restated in the notation of \cite{AJKS} as follows. For an ordered matching $\HH$, let $L(\HH)$ be the size of the largest $P$-clique in $\HH$ for any pattern $P$. Define a Ramsey function $L_r(n)$ to be the minimum of $L(\HH)$ over all $r$-uniform ordered hypergraph matchings $\HH$ with $n$ edges. Then our Theorem~\ref{thm-main} can be restated as $L_r(n) \ge C_r^{-1/(2^r-1)}\cdot n^{1/(2^r-1)}$. Comparing this with the construction in \cite{AJKS} giving $L_r(n) \le \lceil n^{1/(2^r-1)} \rceil$, we obtain
$$
\lceil n^{1/(2^r-1)} \rceil \ge L_r(n) \ge C_r^{-1/(2^r-1)}\cdot n^{1/(2^r-1)}  > \frac{1}{2}n^{1/(2^r-1)}.
$$
So this determines $L_r(n)$ up to a constant of at most $2$ for all $r$ and $n$ (and the constant factor $C_r^{-1/(2^r-1)}=(2^{r-1}(r-1)!)^{-1/(2^r-1)}$ actually approaches 1 as $r$ tends to infinity).

In Theorem~\ref{thm-main}, for every $r$-pattern $P$, the size of every $P$-clique contained in $\HH$ is bounded by $m$. Dudek--Grytczuk--Ruci\'{n}ski and Anastos--Jin--Kwan--Sudakov also asked for the maximum size of $|\HH|$ when bounding the sizes of the $P$-cliques in $\HH$ by different numbers for different $r$-patterns $P$. In other words, how large can $|\HH|$ be, if for each $r$-pattern $P$ the size of each $P$-clique in $\HH$ is at at most a given number $m_P$? We also answer this more general question up to a constant factor depending on $r$.

To state the answer to this multi-parameter version of the problem, we need some more notation. First, it turns out that not for every $r$-pattern $P$ it is actually possible to form a large $P$-clique at all. One can show that this is only possible if $P$ has a particular block structure, where the string $P\in\{{\rm A},{\rm B}\}^{2r}$ can be divided into (consecutive) blocks of even size such that within each block the string reads either $\rm{AA}\dots\rm{A}\rm {BB}\dots\rm{B}$ with equally many $\rm{A}$'s and $\rm{B}$'s or $\rm{BB}\dots\rm{B}\rm {AA}\dots\rm{A}$ with equally many $\rm{A}$'s and $\rm{B}$'s. Following \cite{AJKS}, we call such a decomposition of an $r$-pattern $P$ a \emph{block partition}. For example, the $5$-pattern $\rm{AABBBABBAA}$ has the block partition $\rm{AABB|BA|BBAA}$, whereas the $5$-pattern $\rm{AABABBABBA}$ does not have a block partition.

It is not hard to see that every $r$-pattern $P$ has at most one block decomposition. Following \cite{DGR}, we call an $r$-pattern with a block decomposition \emph{collectable} (then it turns out that the collectable $r$-patterns are precisely the $r$-patterns $P$ for which there exist arbitrarily large $P$-cliques, this was shown in \cite{DGR}).

Given a positive integer $r$, we call a sequence $\lambda=(\lambda_1,\dots,\lambda_s)$ of positive integers with $\lambda_1+\dots+\lambda_s=r$ an \emph{ordered partition of $r$}. One may think of $\lambda$ as a partition of the set $\{1,\dots,r\}$ into non-empty intervals. We call $s$ the \emph{number of parts} of the ordered partition $\lambda=(\lambda_1,\dots,\lambda_s)$. For two ordered partitions $\lambda=(\lambda_1,\dots,\lambda_s)$ and $\lambda'=(\lambda'_1,\dots,\lambda'_{s'})$ of $r$, we write $\lambda \succ \lambda'$ if the identity $\lambda'_1+\dots+\lambda'_{s'}=r$ can be obtained from $\lambda_1+\dots+\lambda_s=r$ by splitting some summands into smaller sub-summands (without changing the order or re-combining any summands). In the language of partitions of $\{1,\dots,r\}$ into intervals, we have $\lambda \succ \lambda'$ if the partition corresponding to $\lambda'$ is a refinement of the partition corresponding to $\lambda$. For example, we have $(2,3) \succ (2,1,2)$ but $(2,3) \not\succ (1,2,2)$.

For any ordered partition $\lambda=(\lambda_1,\dots,\lambda_s)$ of $r$, let $\PP(\lambda)$ be the set of collectable $r$-patterns $P$ such that the block partition of $P$ consists of $s$ blocks of sizes $2\lambda_1,2\lambda_2,\dots,2\lambda_s$ (in this order). For instance, we have 
\[
\PP((1,2,1)) = \{\rm{ABAABBAB, ABAABBBA, ABBBAAAB, ABBBAABA}\}.
\]
Note that the sets $\PP(\lambda)$ for all ordered partitions $\lambda$ or $r$ form a partition of the set of all collectable $r$-patterns.
With this notation, our multi-parameter generalization of Theorem~\ref{thm-main} can be stated as follows. 

\begin{theorem} \label{thm-paramteters-upper-bound}
Let $r$ be a positive integer, and consider a positive integer $m_P$ for every collectable $r$-pattern $P$. Let $\HH$ be an ordered $r$-uniform hypergraph matching, and assume that for every collectable $r$-pattern $P$, every $P$-clique contained in $\HH$ has size at most $m_P$. Then the number of edges in $\HH$ is bounded by
\begin{equation}\label{eq_main_upper}
|\HH|\le 2^{r-1} \sum_{\lambda^{(1)} \succ \lambda^{(2)} \succ \dots \succ \lambda^{(r)}}\ \prod_{P\in \PP(\lambda^{(1)})\cup\dots\cup \PP(\lambda^{(r)})} m_{P},
\end{equation}
where the sum is over all sequences $\lambda^{(1)} \succ \lambda^{(2)} \succ \dots \succ \lambda^{(r)}$ of ordered partitions of $r$ (note that then for $s=1,\dots,r$, the ordered partition $\lambda^{(s)}$ must automatically have $s$ parts).
\end{theorem}

Note that Theorem~\ref{thm-paramteters-upper-bound} directly implies Theorem~\ref{thm-main} by setting $m_P=m$ for all collectable $r$-patterns $P$. Indeed, for every ordered partition $\lambda=(\lambda_1,\dots,\lambda_s)$ of $r$, we have $|\PP(\lambda)|=2^{s-1}$ (for each block of $P \in \PP(\lambda)$ except the first we have two choices for the ordering of A's and B's). There are $(r-1)!$ different sequences $\lambda^{(1)} \succ \lambda^{(2)} \succ \dots \succ \lambda^{(r)}$ of ordered partitions of $r$ (since $\lambda^{(r)}=(1,\dots,1)$ is the unique ordered partition of $r$ with $r$ parts, and for any partition $\lambda^{(s)}$ with $s$ parts there are precisely $s-1$ partitions $\lambda^{(s-1)} \succ \lambda^{(s)}$ with $s-1$ parts), and for each of them we have $|\PP(\lambda^{(1)})\cup\dots\cup \PP(\lambda^{(r)})|=1+2+\dots+2^{r-1}=2^r-1$. Hence the right-hand side of (\ref{eq_main_upper}) simplifies to $2^{r-1}(r-1)!\cdot m^{2^r-1}$ if $m_P=m$ for all $P$.

Our next result shows that the bound in Theorem~\ref{thm-paramteters-upper-bound} is tight up to a constant factor depending on $r$.

\begin{theorem} \label{thm-paramteters-lower-bound}
Let $r$ be a positive integer, and consider a positive integer $m_P$ for every collectable $r$-pattern $P$. Then there exists an  ordered $r$-uniform hypergraph matching $\HH$, such that for every collectable $r$-pattern $P$, every $P$-clique contained in $\HH$ has size at most $m_P$, and such that
\begin{equation}\label{eq_main_lower}
|\HH|= \max_{\lambda^{(1)} \succ \lambda^{(2)} \succ \dots \succ \lambda^{(r)}}\ \prod_{P\in \PP(\lambda^{(1)})\cup\dots\cup \PP(\lambda^{(r)})} m_{P},
\end{equation}
where the maximum is over all sequences $\lambda^{(1)} \succ \lambda^{(2)} \succ \dots \succ \lambda^{(r)}$ of ordered partitions of $r$.
\end{theorem}

Note that the right-hand side of \eqref{eq_main_lower} differs from the right-hand side of \eqref{eq_main_upper} by a factor of at most $2^{r-1}(r-1)!$. Indeed, there are exactly $(r-1)!$ different sequences $\lambda^{(1)} \succ \lambda^{(2)} \succ \dots \succ \lambda^{(r)}$, and \eqref{eq_main_lower} describes the largest of the $(r-1)!$ summands appearing on the right-hand side of \eqref{eq_main_upper}.

Thus, Theorems~\ref{thm-paramteters-upper-bound} and~\ref{thm-paramteters-lower-bound} determine, up to a factor of at most $2^{r-1}(r-1)!$, the maximum size of $|\HH|$ if for every collectable $r$-pattern every $P$-clique in $\HH$ has at most some given size $m_P$.

After discussing some preliminaries in Section~\ref{sect-preliminaries}, we will prove Theorem~\ref{thm-paramteters-upper-bound} (which implies Theorem~\ref{thm-main}) in Section~\ref{sect-upper-bound}, and Theorem~\ref{thm-paramteters-lower-bound} in Section~\ref{sect-lower-bound}.

\section{Preliminaries}
\label{sect-preliminaries}

The Erd\H{o}s--Szekeres Theorem~\cite{ES} states that any sequence of $n^2+1$ distinct integers contains a monotone subsequence of length $n+1$. This is a classical result in combinatorics with a rich family of generalizations and connections, see for example~\cite{ExpES,GL,LS,MS,ST}.

We will need a multi-dimensional version of the the classical (one-dimensional) Erd\H{o}s--Szekeres Theorem. For our purposes, this multi-dimensional version is most convenient to state in the following way. For a function $\tau:\{1,\dots,s\}\to \{1,-1\}$, let us say that a sequence of points $z_1,\dots,z_m\in \mathbb{R}^s$ is \emph{$\tau$-monotone} if for every $i\in \{1,\dots,s\}$ with $\tau(i)=1$ the $i$-th coordinates of the points $z_1,\dots,z_m$ form a strictly increasing sequence and for every $i\in \{1,\dots,s\}$ with $\tau(i)=-1$ the $i$-th coordinates of the points $z_1,\dots,z_m$ form a strictly decreasing sequence. 

\begin{theorem}[Multi-dimensional Erd\H{o}s--Szekeres Theorem]\label{thm_es}
Let $s$ be a positive integer. Furthermore, consider a positive integer $m_\tau$ for every function $\tau:\{1,\dots,s\}\to \{1,-1\}$ with $\tau(1)=1$. Let $Z\su \mathbb{R}^s$ be a collection of points, whose $|Z|\cdot s$ coordinates are all distinct. For each $\tau:\{1,\dots,s\}\to \{1,-1\}$ with $\tau(1)=1$, assume that every $\tau$-monotone sequence of points in $Z$ has length at most $m_\tau$. Then the number of points in $Z$ is bounded by
\[|Z|\le \prod_{\tau} m_{\tau},\]
where the product is over all functions $\tau:\{1,\dots,s\}\to \{1,-1\}$ with $\tau(1)=1$.

Furthermore, for any choice of positive integers $m_\tau$ for all functions $\tau:\{1,\dots,s\}\to \{1,-1\}$ with $\tau(1)=1$, there exists a collection of points $Z\su \mathbb{R}^s$ of size $|Z|=\prod_\tau m_\tau$ satisfying the conditions above.
\end{theorem}

Note that if $z_1, \ldots, z_m$ is a $\tau$-monotone sequence of points in $Z\su \mathbb{R}^s$ for some function $\tau: \{1,\dots,s\}\to \{1,-1\}$, then the sequence $z_m, \ldots, z_1$ is $(-\tau)$-monotone. So the largest $\tau$-monotone and $(-\tau)$-monotone sequences in $Z$ have the same size. Therefore it suffices to only consider functions $\tau$ satisfying $\tau(1)=1$ in the statement of Theorem~\ref{thm_es}.

The classical Erd{\H o}s--Szekeres theorem states that a sequence $a_1, \ldots, a_{n}$ of distinct integers must have length $n\le mm'$ if its longest monotone increasing subsequence has length at most $m$ and its longest monotone increasing subsequence has length at most $m'$. This can be viewed as a special case of Theorem~\ref{thm_es} by taking $Z\su \mathbb{R}^2$ to consist of the points $(i,a_i)$ for $i=1, \ldots, n$. Then monotone increasing subsequences of $a_1, \ldots, a_{n}$ correspond to $\tau$-monotone sequences in $Z$ for the function $\tau:\{1, 2\} \rightarrow \{1,-1\}$ given by $\tau(1)=\tau(2)=1$, and monotone decreasing subsequences of $a_1, \ldots, a_{n}$ correspond to $\tau$-monotone sequences in $Z$ for the function $\tau:\{1, 2\} \rightarrow \{1,-1\}$ given by $\tau(1)=1$ and $\tau(2)=-1$.

Theorem~\ref{thm_es} was essentially proven in \cite{DGR} and used to construct the ordered hypergraph matchings with cliques of bounded sizes. We include a simple (and fairly standard) proof for the reader's convenience.

\begin{proof}
Note that each function $\tau:\{1, \ldots, s\}\rightarrow \{1,-1\}$ with $\tau(1) = 1$ naturally defines a partial order $\prec_{\tau}$ on $Z$. Let $\tau_1, \ldots, \tau_{2^{s-1}}$ be the list of all such functions in an arbitrary order. Then we are given positive integers $m_{\tau_i}$ for $i=1,\dots,2^{s-1}$.

Let $Z \subseteq \mathbb R^s$ be as in the first part of the theorem statement. By the assumption, for $i=1,\dots,2^{s-1}$, the length of any chain in $(Z, \prec_{\tau_i})$ is at most $m_{\tau_i}$. 
 By Dilworth's Theorem~\cite{Dil} applied to $(Z, \prec_{\tau_1})$, there exists an $\prec_{\tau_1}$-antichain $Z_1 \subseteq Z$ of size at least $|Z| / m_{\tau_1}$. Then by Dilworth's Theorem applied to $(Z_1, \prec_{\tau_2})$, there exists an $\prec_{\tau_2}$-antichain $Z_2 \subseteq Z_1$ of size at least $|Z_1| / m_{\tau_2} \ge |Z|/(m_{\tau_1}m_{\tau_2})$. Continuing this process for all $i=1, \ldots, 2^{s-1}$, we obtain a subset
\[Z' = Z_{2^{s-1}} \subseteq \ldots \subseteq Z_1\su Z\]
of size $|Z'|\ge |Z|/\prod_{i=1}^{2^{s-1}} m_{\tau_i}$ which is an antichain with respect to $\prec_{\tau_i}$ for $i=1,\dots,2^{s-1}$. However, any two vectors in $\mathbb R^s$ with distinct coordinates are comparable with respect to exactly one of the partial orders $\prec_{\tau_1},\dots,\prec_{\tau_{2^{s-1}}}$. Thus, we must have $|Z'|\le 1$, implying $|Z|\le \prod_{i=1}^{2^{s-1}} m_{\tau_i}=\prod_\tau m_{\tau}$, as desired. 

Now we prove the second part of the theorem, i.e.\ we construct a collection of points $Z\su \mathbb{R}^s$ achieving equality in the bound above. Let $N$ be a sufficiently large integer such that $N\ge 3m_{\tau_i}$ for $i=1,\dots,2^{s-1}$. For a sequence of numbers $a = (a_1, \ldots, a_{2^{s-1}})$, where $a_i \in \{1, \ldots, m_{\tau_i}\}$ for $i=1, \ldots, 2^{s-1}$, we define a vector $z(a) = (z(a)_1, \ldots, z(a)_s) \in \mathbb R^s$ by setting
\begin{equation}\label{eq_z}
z(a)_j = \sum_{i=1}^{2^{s-1}} N^{i} \tau_i(j) a_{i}
\end{equation}
for $j=1,\dots,s$. Let $Z \subseteq \mathbb R^s$ be the collection of all points $z(a)$ for all tuples $a = (a_1, \ldots, a_{2^{s-1}})$ with $a_i \in \{1, \ldots, m_{\tau_i}\}$ for $i=1, \ldots, 2^{s-1}$. First note that by the choice of $N$, all the sums of the form (\ref{eq_z}) are pairwise distinct for all possible choices of sequences $a$ and all $j=1,\dots,s$. So points the $|Z|\cdot s$ coordinates of the points $Z$ are all distinct coordinates, and we have $|Z| = \prod_{i=1}^{2^{s-1}} m_{\tau_i}=\prod_\tau m_\tau$. It remains to show that for any $i = 1, \ldots, 2^{s-1}$, any $\prec_{\tau_i}$-chain in $Z$ has length at most $m_{\tau_i}$. 

So let us some $k \in \{1, \ldots, 2^{s-1}\}$, and suppose that for some sequences $a^{(1)}, \ldots, a^{(t)}$ as above, the points $z(a^{(1)}), \ldots, z(a^{(t)}) \in Z$ form a $\tau_k$-monotone sequence. That means that for each $j \in \{1, \ldots, s\}$ we have $\tau_k(j) z(a^{(1)})_j<\dots<\tau_k(j) z(a^{(t)})_j$. 
Pick an arbitrary pair of indices $1\le \ell < \ell' \le t$ and let $k' \in \{1, \ldots, 2^{s-1}\}$ be the maximum index such that $a^{(\ell)}_{k'} \neq a^{(\ell')}_{k'}$.

Note that by (\ref{eq_z}), for any $j\in \{1, \ldots, s\}$ we have 
$$
z(a^{(\ell)})_j - z(a^{(\ell')})_j = \sum_{i=1}^{k'} N^i \tau_i(j) (a^{(\ell)}_i - a^{(\ell')}_i).
$$
The $k'$-th term in the sum above is larger in absolute value than all the smaller terms combined, which implies that the numbers $z(a^{(\ell)})_j - z(a^{(\ell')})_j$ and $\tau_{k'}(j)(a^{(\ell)}_{k'} - a^{(\ell')}_{k'})$ have the same sign. On the other hand, we know that $\tau_k(j) z(a^{(\ell)})_j < \tau_k(j) z(a^{(\ell')})_j$. So we conclude that 
\begin{equation}\label{taus}
\tau_{k}(j) \tau_{k'}(j) a^{(\ell)}_{k'} < \tau_{k}(j) \tau_{k'}(j) a^{(\ell')}_{k'}    
\end{equation}
holds for every $j \in \{1, \ldots, s\}$. Letting $j = 1$ and using $\tau_k(1) = \tau_{k'}(1)=1$ implies that $a^{(\ell)}_{k'} < a^{(\ell')}_{k'}$. Combining this with (\ref{taus}), for $j=2,\dots,s$ we get $\tau_k(j) = \tau_{k'}(j)$, meaning that the functions $\tau_k$ and $\tau_{k'}$ coincide and we in fact have $k = k'$.

We conclude that for any $1\le \ell<\ell' \le t$, the last coordinate where the sequences $a^{(\ell)}$ and $a^{(\ell')}$ differ is always equal to $k$. In particular, all $k$-th coordinates of the sequences $a^{(1)}, \ldots, a^{(t)}$ are pairwise distinct. Since these coordinates take values in the set $\{1, \ldots, m_{\tau_k}\}$, we can conclude that $t \le m_{\tau_k}$. This shows that the longest $\tau_k$-monotone sequence in $Z$ has length at most $m_{\tau_k}$, as desired.
\end{proof}

In order to make use of Theorem~\ref{thm_es} in our context, the following notation will be helpful. Consider an ordered partition $\lambda=(\lambda_1,\dots,\lambda_s)$ of $r$. Then for any function $\tau:\{1,\dots,s\}\to \{1,-1\}$ with $\tau(1)=1$, we define a collectable $r$-pattern $P(\lambda,\tau)\in \PP(\lambda)$ as follows. The block partition of $P$ consists of $s$ blocks of sizes $2\lambda_1,\dots,2\lambda_s$ (in this order). For each $i\in\{1,\dots,s\}$ with $\tau(i)=1$, in the $i$-th block we take $\lambda_i$ times the letter $\rm A$ followed by $\lambda_i$ times the letter $\rm B$. And for each $i\in\{1,\dots,s\}$ with $\tau(i)=-1$, in the $i$-th block we take $\lambda_i$ times the letter $\rm B$ followed by $\lambda_i$ times the letter $\rm A$. This gives a valid block decomposition of an $r$-pattern, and we call the resulting $r$-pattern $P(\lambda,\tau)$. Note that by definition, we have $P(\lambda,\tau)\in \PP(\lambda)$ and in particular the $r$-pattern $P(\lambda,\tau)$ is collectable. Furthermore, every $r$-pattern $P\in \PP(\lambda)$ is of the form $P=P(\lambda,\tau)$ for some function $\tau:\{1,\dots,s\}\to \{1,-1\}$ with $\tau(1)=1$.

Another helpful piece of notation for us is to write $\conv(S)$ for the convex hull of a finite subset $S\su \mathbb{R}$ (note that this convex hull is simply the interval $[\min S,\max S]$ between the minimium and the maximum of $S$).

\section{Proof of the upper bound}
\label{sect-upper-bound}

Every ordered $r$-uniform hypergraph matching can be realized by assigning (positive) integers to the vertices and taking the total ordering on the vertex set to be the natural ordering of the integers. In order to prove Theorem~\ref{thm-paramteters-upper-bound}, we can therefore assume without loss of generality that our ordered $r$-uniform hypergraph matching $\HH$ is given in this way. In other words, we can take $\HH$ to be a finite collection of pairwise disjoint subsets $e\su \Z$, each of size $|e|=r$. This somewhat simplifies the notation in our proof.

As observed for example in \cite{AJKS}, for ordered $r$-uniform hypergraph matching $\HH$ with a certain $r$-partite structure, one can bound the size $|\HH|$ in terms of the maximum sizes of the $P$-cliques contained in $|\HH|$ for each $r$-pattern $P$, just by using the multi-dimensional Erd\H{o}s--Szekeres Theorem. The following definition describes a more general condition under which this is possible.

\begin{definition}\label{def-s-partite}
    Let $r$ be a positive integer, and consider a finite collection $\HH$ of pairwise disjoint subsets $e\su \Z$, each of size $|e|=r$. For an ordered partition $\lambda=(\lambda_1,\dots,\lambda_s)$ of $r$, we say that $\HH$ is \emph{$\lambda$-partite} if there are $x_0,x_1,\dots,x_{s}\in\mathbb{R}\setminus\Z$ with $x_0<x_1<\dots<x_{s}$ such that we have $|e\cap (x_{i-1},x_i)|=\lambda_i$ for every $e\in \HH$ and every $i=1,\dots,s$. We say that $\HH$ is \emph{interval-wise $\lambda$-partite} if $x_0<x_1<\dots<x_{s}$ can be chosen such that in addition to the previous condition, for each $i=1,\dots,s$, the intervals $\conv(e\cap (x_{i-1},x_i))$  are pairwise disjoint for all $e\in \HH$.
\end{definition}

Intuitively speaking, for an ordered partition $\lambda=(\lambda_1,\dots,\lambda_s)$ of $r$, saying that $\HH$ is $\lambda$-partite means that there are $s$ intervals $(x_0,x_1),(x_1,x_2),\dots,(x_{s-1},x_s)$ such that every $e\in \HH$ has exactly $\lambda_1$ elements in the first interval $(x_0,x_1)$, exactly $\lambda_2$ elements in the second interval $(x_1,x_2)$, and so on. For $\HH$ to be interval-wise $\lambda$-partite means that the sub-intervals of $(x_0,x_1)$ spanned by the intersections $e\cap (x_0,x_1)$ are pairwise disjoint for all $e\in \HH$, and similarly for the sub-intervals of $(x_1,x_2)$ spanned by the intersections $e\cap (x_1,x_2)$, and so on.

For example, when taking $r=3$ and $\lambda=(1,2)$, then
\[
\HH = \{ \{1,4,7\}, \{2, 5,6\}, \{3,8,9\} \}
\]
is $\lambda$-partite (taking $x_0=0.5$, $x_1 = 3.5$, and $x_2 = 9.5$), but not interval-wise $\lambda$-partite. Furthermore,
\[
\mathcal H' = \{\{1,6,7\}, \{2, 4,5\}, \{3,8,9\}\},
\]
is interval-wise $\lambda$-partite (again taking $x_0=0.5$, $x_1 = 3.5$, and $x_2 = 9.5$). See also Figure~\ref{fig:figure-examples-partite} for an illustration.

\begin{figure}
    \centering
    \begin{tikzpicture}[x=0.6cm,y=0.6cm]
    \filldraw[red] (1,3) circle (3pt);
    \filldraw[red] (4,3) circle (3pt);
    \filldraw[red] (7,3) circle (3pt);
    \filldraw[blue] (2,2) circle (3pt);
    \filldraw[blue] (5,2) circle (3pt);
    \filldraw[blue] (6,2) circle (3pt);
    \filldraw[Green] (3,1) circle (3pt);
    \filldraw[Green] (8,1) circle (3pt);
    \filldraw[Green] (9,1) circle (3pt);
    \draw[dotted] (0.5,0.5) -- (0.5,3.5);
    \draw[dotted] (3.5,0.5) -- (3.5,3.5);
    \draw[dotted] (9.5,0.5) -- (9.5,3.5);
    \end{tikzpicture}
    \hspace{3cm}
    \begin{tikzpicture}[x=0.6cm,y=0.6cm]
    \filldraw[Plum] (1,3) circle (3pt);
    \filldraw[Plum] (6,3) circle (3pt);
    \filldraw[Plum] (7,3) circle (3pt);
    \draw[Plum] (6,3)--(7,3);
    \filldraw[Orange] (2,2) circle (3pt);
    \filldraw[Orange] (4,2) circle (3pt);
    \filldraw[Orange] (5,2) circle (3pt);
    \draw[Orange] (4,2)--(5,2);
    \filldraw[Cyan] (3,1) circle (3pt);
    \filldraw[Cyan] (8,1) circle (3pt);
    \filldraw[Cyan] (9,1) circle (3pt);
    \draw[Cyan] (8,1)--(9,1);
    \draw[dotted] (0.5,0.5) -- (0.5,3.5);
    \draw[dotted] (3.5,0.5) -- (3.5,3.5);
    \draw[dotted] (9.5,0.5) -- (9.5,3.5);
    \end{tikzpicture}
    \caption{The left side shows $\HH = \{ \{1,4,7\}, \{2, 5,6\}, \{3,8,9\}\}$, which is $\lambda$-partite for $\lambda=(1,2)$. The right side shows $\mathcal H' = \{\{1,6,7\}, \{2, 4,5\}, \{3,8,9\}\}$, which is interval-wise $\lambda$-partite (again for $\lambda=(1,2)$). In both pictures, the rows correspond to the individual sets in $\HH$ and  $\mathcal H'$, respectively.}
    \label{fig:figure-examples-partite}
\end{figure}


Note that $\HH$ is always $\lambda$-partite for the ordered partition $\lambda=(r)$ of $r$ into a single part (taking $(x_0,x_1)$ to be some interval containing all elements of all $e\in \HH$). Also note that if $\HH$ is $\lambda$-partite for some ordered partition $\lambda$, then $\HH$ is automatically also $\lambda'$-partite for all $\lambda'$ with $\lambda'\succ \lambda$. Furthermore, for the ordered partition $\lambda=(1,1,\dots,1)$ of $r$ into $r$ parts, it is equivalent for $\HH$ to be $\lambda$-partite or to be interval-wise $\lambda$-partite.

Finally, if $\HH$ is interval-wise $\lambda$-partite for some ordered partition $\lambda$ of $r$, then any two distinct $e,f\in\HH$ form some $r$-pattern in $\PP(\lambda)$. Which of the $r$-patterns in $\PP(\lambda)$ is formed by $e$ and $f$ depends on the the order of the intervals $\conv(e\cap (x_{i-1},x_i))$ and $\conv(f\cap (x_{i-1},x_i))$ appearing in Definition~\ref{def-s-partite} (for each $i$, there are two possibilities which of the two intervals comes first).

If $\HH$ is interval-wise $\lambda$-partite for some ordered partition $\lambda$ of $r$, then a bound on $|\HH|$ under the assumptions in Theorem~\ref{thm-paramteters-upper-bound} follows directly from the multi-dimensional Erd\H{o}s--Szekeres Theorem (see Theorem~\ref{thm_es}). 

\begin{lemma}\label{obs-interval-wise}
Let $r$ be a positive integer, and let $\lambda=(\lambda_1,\dots,\lambda_s)$ be an ordered partition of $r$. Furthermore, consider a positive integer $m_P$ for every $r$-pattern $P\in \PP(\lambda)$. Let $\HH$ be a finite collection of pairwise disjoint subsets $e\su \Z$, each of size $|e|=r$, such that $\HH$ is interval-wise $\lambda$-partite. For every $P\in \PP(\lambda)$, assume that every $P$-clique contained in $\HH$ has size at most $m_{P}$. Then we have
\[|\HH|\le \prod_{P\in \PP(\lambda)} m_{P}.\]
\end{lemma}

\begin{proof} 
    As $\HH$ is interval-wise $\lambda$-partite, there exist $x_0<x_1<\dots<x_{s}$ such that for each $i=1,\dots,s$, we have $|e\cap (x_{i-1},x_i)|=\lambda_i>0$ for every $e\in \HH$ and the intervals $\conv(e\cap (x_{i-1},x_i))$ are pairwise disjoint for all $e\in \HH$. Now, for every $e\in \HH$, let us define a point $z(e)\in \mathbb{R}^s$ by taking
    \[z(e)=(\min(e\cap (x_{0},x_1)), \min(e\cap (x_{1},x_2)), \dots, \min(e\cap (x_{s-1},x_s))),\]
    and define $Z=\{ z(e)\mid e\in \HH\}\su \mathbb{R}^s$. As the sets $e\in \HH$ are pairwise disjoint, the all the coordinates of the points $z(e)$ for $e\in \HH$ are distinct.
    
    For every function $\tau:\{1,\dots,s\}\to \{1,-1\}$ with $\tau(1)=1$, let us define $m_\tau=m_{P(\lambda,\tau)}$ (recall that $P(\lambda,\tau)\in \PP(\lambda)$ was defined at the end of Section~\ref{sect-preliminaries}). We claim that then every  $\tau$-monotone sequence of points in $Z\su \mathbb{R}^s$ has length at most $m_\tau$.

    Indeed, for any $\rho:\{1,\dots,s\}\to \{1,-1\}$ with $\tau(1)=1$, consider $e^{(1)},e^{(2)},\dots,e^{(t)}\in \HH$ such that the point sequence $z(e^{(1)}),z(e^{(2)}),\dots,z(e^{(t)})$ is a $\tau$-monotone sequence in $Z \su \mathbb{R}^s$. Then for any index $i\in \{1,\dots,s\}$ with $\tau(i)=1$ we have  $\min(e^{(1)}\cap (x_{i-1},x_i))< \dots <\min(e^{(t)}\cap (x_{i-1},x_i))$, and in particular $\min(e^{(\ell)}\cap (x_{i-1},x_i))<\min(e^{(\ell')}\cap (x_{i-1},x_i))$ for any $1\le \ell<\ell'\le t$. Since furthermore the intervals $\conv(e^{(\ell)}\cap (x_{i-1},x_i))$ and $\conv(e^{(\ell')}\cap (x_{i-1},x_i))$ are disjoint, this means that all elements of $e^{(\ell)}\cap (x_{i-1},x_i)$ are smaller than all elements of $e^{(\ell')}\cap (x_{i-1},x_i)$. Similarly, for any $i\in \{1,\dots,s\}$ with $\tau(i)=-1$ and any $1\le \ell<\ell'\le t$, all elements of $e^{(\ell)}\cap (x_{i-1},x_i)$ are larger than all elements of $e^{(\ell')}\cap (x_{i-1},x_i)$. Since $x_0<x_1<\dots<x_{s}$ and $|e^{(\ell)}\cap (x_{i-1},x_i)|=|e^{(\ell')}\cap (x_{i-1},x_i)|=\lambda_i$ for all  $i=1,\dots,s$, this shows that  $e^{(\ell)}$ and $e^{(\ell')}$ form the pattern $P(\lambda,\tau)$ for all $1\le \ell<\ell'\le t$. Hence $e^{(1)},e^{(2)},\dots,e^{(t)}\in \HH$ form a $P(\lambda,\tau)$-clique and therefore $t\le m_{P(\lambda,\tau)}=m_\rho$. So every $\rho$-monotone sequence in $Z$ indeed has length at most $m_\rho$.

    Applying Theorem~\ref{thm_es} now gives
    \[|\HH|=|Z|\le \prod_{\substack{\tau:\{1,\dots,s\}\to \{1,-1\}\\ \tau(1)=1}} m_{\tau}=\prod_{\substack{\tau:\{1,\dots,s\}\to \{1,-1\}\\ \tau(1)=1}} m_{P(\lambda,\tau)}=\prod_{P\in \PP(\lambda)} m_{P},\]
    as desired.
\end{proof}

The main step in our proof of Theorem~\ref{thm-paramteters-upper-bound} is showing the following proposition, which roughly speaking states that for every large $\lambda$-partite $\HH$ we can find a large $\HH^*\su \HH$ which is interval-wise $\lambda$-partite or we can find a large $\HH'\su \HH$ which is $\lambda'$-partite for some ordered partition $\lambda'$ of $r$ with $\lambda\succ \lambda'$. In the first case, we can then apply Lemma~\ref{obs-interval-wise}, whereas in the second case we can repeat the argument with $\lambda'$ instead of $\lambda$.

\begin{proposition}\label{prop-two-options}
Let $r$ be a positive integer, and let $\lambda=(\lambda_1,\dots,\lambda_s)$ be an ordered partition of $r$ into $s<r$ parts. For every ordered partition $\lambda'$ of $r$ into $s+1$ parts with $\lambda\succ\lambda'$, consider some positive integer $M_{\lambda'}$, and let $M=\sum_{\lambda'} M_{\lambda'}$ be the sum of all these positive integers. Let $\HH$ be a finite collection of pairwise disjoint subsets $e\su \Z$, each of size $|e|=r$, such that $\HH$ is $\lambda$-partite. Then at least one of the following two statements holds:
\begin{itemize}
    \item[(a)] There is a sub-collection $\HH^*\su \HH$ of size $|\HH^*|\ge |\HH|/(2M)$ such that $\HH^*$ is interval-wise $\lambda$-partite.
    \item[(b)] For some ordered partition $\lambda'$ of $r$ into $s+1$ parts with $\lambda\succ\lambda'$, there is a sub-collection $\HH'\su \HH$ of size $|\HH'|>M_{\lambda'}$ such that $\HH'$ is $\lambda'$-partite.
\end{itemize}
\end{proposition}

Let us illustrate the proof idea for Proposition~\ref{prop-two-options} in the special case $\lambda = (r)$. Given $\HH$, consisting of pairwise disjoint subsets $e\su \Z$ of size $|e|=r$, we consider the collection of intervals $\conv(e)$ for all $e \in \HH$. We observe that at least one one of the following two statements holds: There exists a sub-collection $\HH^*\su \HH$ of size $|\HH^*|\ge |\HH|/M$ such that the intervals $\conv(e)$ for $e\in \HH^*$ are pairwise disjoint, or there exists a point $y \in \mathbb R\setminus \Z$ covered by at least $M+1$ of the intervals $\conv(e)$ for $e\in \HH$. This is a general fact about an arbitrary collection of intervals and can be easily proven by induction. In the first alternative, the sub-collection $\HH^*\su \HH$ is interval-wise $\lambda$-partite, so (a) holds (even with a slightly better lower bound on $|\HH^*|$). In the second alternative, consider distinct $e_1, \ldots, e_{M+1}\in \HH$ whose corresponding intervals $\conv(e_i)$ contain $y$. Since $M = \sum_{j=1}^{r-1} M_{(j, r-j)}$, by the pigeonhole principle, there exists some $j \in \{1, \ldots, r-1\}$ such that there are at least $M_{(j, r-j)}+1$ different $e_i$ with $|e_i \cap (-\infty, y)| = j$. The collection of these sets $e_i$ now forms a $(j, r-j)$-partite sub-collection $\HH'\su \HH$ of size $|\HH'|>M_{(j, r-j)}$ (indeed, to see that $\HH'$ is $(j, r-j)$-partite, one can choose $x_0<\min \bigcup_{e\in \HH} e$ and $x_1=y$ and $x_2>\max \bigcup_{e\in \HH} e$, observing that then $|e_i \cap (x_0, x_1)| = j$ and $|e_i \cap (x_1, x_2)| = r-j$ for all $e_i\in \HH'$), giving (b). The proof of Proposition~\ref{prop-two-options} for general $\lambda$ is similar to this special case, but involves a bit more notation.  In particular, we need to deal with unions of multiple intervals for each $e\in \HH$ instead of simply considering $\conv(e)$, and to use a degree counting argument (with a factor $2$ loss in the bound) to get a similar conclusion in (a).


\begin{proof}[Proof of Proposition~\ref{prop-two-options}]
Suppose that $\HH$ does not have a subcollection $\HH'$ as in (b). We need to show that then (a) holds. Let $x_0<x_1<\dots<x_s$ be as in Definition~\ref{def-s-partite}. For every $e\in \HH$, we consider the intervals $e\cap (x_{i-1},x_i)$ for $i=1,\dots,s$. Our goal is to find a subcollection $\HH^*\su \HH$ of size $|\HH^*|\ge |\HH|/(2M)$ such that for each $i=1,\dots,s$, the intervals $\conv(e\cap (x_{i-1},x_i))$ are pairwise disjoint for all $e\in \HH^*$. Note that for any $i=1,\dots,s$ and any $e\in \HH$, the interval $\conv(e\cap (x_{i-1},x_i))$ is a closed sub-interval of $(x_{i-1},x_i)$. Furthermore, the endpoints of all these sub-intervals are distinct (as all $e\in \HH$ are pairwise disjoint subsets of $\Z$).

Let us consider an auxiliary graph $G$ with vertex set $\HH$, where for $e,e'\in\HH$ we draw an edge between $e$ and $e'$ if $\conv(e\cap (x_{i-1},x_i))\cap \conv(e'\cap (x_{i-1},x_i))\ne \emptyset$ for some $i\in\{1,\dots,s\}$. Our goal is then to show that the graph $G$ has an independent set of size at least $|\HH|/(2M)$.

Note that for any two closed intervals $[a,b]$ and $[c,d]$ with $[a,b]\cap [c,d]\ne \emptyset$ we have $a\in [c,d]$ or $c\in [a,b]$. Hence we can orient the edges of $G$ in such a way that for any $e,e'\in\HH$ with an edge from $e$ to $e'$ there is an index $i\in\{1,\dots,s\}$ with $\min(e\cap (x_{i-1},x_i))\in \conv(e'\cap (x_{i-1},x_i))$.

In order to show that the graph $G$ contains an independent set of size at least $|\HH|/(2M)$, it suffices to show that $G$ has average degree at most $2M-1$ (by a well-known result of Caro~\cite{Caro} and Wei~\cite{Wei}). Considering the orientation of the edges of $G$, it therefore suffices to show that for every $e\in \HH$ there are at most $M-1$ outgoing edges at $e$ in the graph $G$ (then there are at most $(M-1)|\HH|$ total edges in $G$ and hence the average degree is at most $2M-2$).

So let us now fix some $e\in \HH$. We need to show that there are at most $M-1$ different $e'\in\HH\sm\{e\}$ such that $\min(e\cap (x_{i-1},x_i))\in \conv(e'\cap (x_{i-1},x_i))$ for some index $i\in\{1,\dots,s\}$. 

For every index $i\in\{1,\dots,s\}$ with $\lambda_i=1$, we have $|e'\cap (x_{i-1},x_i)|=\lambda_i=1$ for every $e'\in \HH$, and the singleton sets $e'\cap (x_{i-1},x_i)$ are all distinct for all $e'\in \HH$ (as all $e'\in \HH$ are pairwise disjoint). Hence we cannot have $\min(e\cap (x_{i-1},x_i))\in \conv(e'\cap (x_{i-1},x_i))$ for any $e'\in \HH\sm\{e\}$ when $\lambda_i=1$.

We claim that for each $i\in\{1,\dots,s\}$ with $\lambda_i\ge 2$, there can be at most $\sum_{\lambda'} M_{\lambda'}-1$ different $e'\in\HH\sm\{e\}$ with $\min(e\cap (x_{i-1},x_i))\in \conv(e'\cap (x_{i-1},x_i))$, where the sum is over all ordered partitions $\lambda'$ of $r$ into $s+1$ parts that can be obtained from $\lambda$ by dividing the $i$-th part $\lambda_i$ into two pieces. This suffices to finish the proof, since then the total number of $e'\in\HH\sm\{e\}$ with $\min(e\cap (x_{i-1},x_i))\in \conv(e'\cap (x_{i-1},x_i))$ for some $i\in\{1,\dots,s\}$ is at most
\[\sum_{\substack{i\in\{1,\dots,s\}\\\lambda_i\ge 2}}\left(\left(\sum_{\substack{\lambda'\text{obtained from }\lambda\text{ by}\\\text{dividing }\lambda_i\text{into two parts}}}\!\!M_{\lambda'}\right)-1\right)=\sum_{\lambda'} M_{\lambda'}-|\{i\in\{1,\dots,s\}\mid \lambda_i\ge 2\}|\le \sum_{\lambda'} M_{\lambda'}-1=M-1,\]
where in the second and third term the sum is over all ordered partitions $\lambda'$ of $r$ into $s+1$ parts with $\lambda\succ\lambda'$ (note that each of these ordered partitions is obtained by dividing some part $\lambda_i$ of $\lambda$ with $\lambda_i\ge 2$ into two parts). Also note that for the inequality sign we have used that there must at least one index $i\in \{1,\dots,s\}$ with $\lambda_i\ge 2$, since $s<r$.

It remains to prove the above claim, so fix some index $i\in\{1,\dots,s\}$ with $\lambda_i\ge 2$. There are $\lambda_i-1$ different ordered partitions $\lambda'$ that can be obtained from $\lambda$ by dividing $\lambda_i$ into two parts, namely the ordered partitons $\mu^{(1)}=(\lambda_1,\dots,\lambda_{i-1},1,\lambda_i-1,\lambda_{i+1},\dots,\lambda_s), \ \mu^{(2)}=(\lambda_1,\dots,\lambda_{i-1},2,\lambda_i-2,\lambda_{i+1},\dots,\lambda_s),\ \dots,\ \linebreak \mu^{(\lambda_i-1)}=(\lambda_1,\dots,\lambda_{i-1},\lambda_i-1,1,\lambda_{i+1},\dots,\lambda_s)$. Defining $y=\min(e\cap (x_{i-1},x_i))\in \Z$, we need to show that there are at most $\sum_{j=1}^{\lambda_i-1} M_{\mu^{(j)}}-1$ different $e'\in\HH\sm\{e\}$ with $y\in \conv(e'\cap (x_{i-1},x_i))$. Note that for every such $e'\in\HH\sm\{e\}$, we also have $y+0.1\in \conv(e'\cap (x_{i-1},x_i))$, since $\conv(e'\cap (x_{i-1},x_i))$ is a closed interval with endpoints in $\Z\sm\{y\}$. We furthermore also have $y+0.1\in \conv(e\cap (x_{i-1},x_i))$ (recalling that $y=\min(e\cap (x_{i-1},x_i))$ and that $e\cap (x_{i-1},x_i)\su \Z$ is a set of size $\lambda_i\ge 2$), which in particular implies $x_{i-1}<y+0.1<x_i$. It now suffices to show that there are at most $\sum_{j=1}^{\lambda_i-1} M_{\mu^{(j)}}$ different $e'\in\HH$ with $y+0.1\in \conv(e'\cap (x_{i-1},x_i))$. For every such $e'$, we have $1\le |e'\cap (x_{i-1},y+0.1)|\le \lambda_i-1$ (recalling that $e'\su \Z$ and $|e'\cap (x_{i-1},x_i)|=\lambda_i$).

For each $j=1,\dots,\lambda_i-1$, there can be at most $M_{\mu^{(j)}}$ different $e'\in\HH$ with $|e'\cap (x_{i-1},y+0.1)|=j$. Indeed, for each such $e'\in\HH$ we have $|e'\cap (y+0.1,x_i)|=|e'\cap (x_{i-1},x_i)|-|e'\cap (x_{i-1},y+0.1)|=\lambda_i-j$. Hence, by considering $x_0,\dots,x_{i-1}, y+0.1,x_i,\dots,x_s\in \mathbb{R}\sm \Z$, we can observe that the collection $\HH'\su \HH$ of all $e'\in\HH$ with $|e'\cap (x_{i-1},y+0.1)|=j$ is $\mu^{(j)}$-partite. So by our assumption that (b) does not hold, there can be at most $M_{\mu^{(j)}}$ different such $e'\in\HH$ (recall that $\mu^{(j)}$ is an ordered partition of $r$ into $s+1$ parts with $\lambda\succ \mu^{(j)}$).

Thus, in total there can be at most $\sum_{j=1}^{\lambda_i-1} M_{\mu^{(j)}}$ different $e'\in\HH$ with $1\le |e'\cap (x_{i-1},y+0.1)|\le \lambda_i-1$, meaning that there are at most $\sum_{j=1}^{\lambda_i-1} M_{\mu^{(j)}}$ different $e'\in\HH$ with $y+0.1\in \conv(e'\cap (x_{i-1},x_i))$. This finishes the proof of the proposition.
\end{proof}

We are now ready to prove Theorem~\ref{thm-paramteters-upper-bound} using Lemma~\ref{obs-interval-wise} and Proposition~\ref{prop-two-options}. In order to do so, we inductively show the following generalization of the theorem.

\begin{theorem} \label{thm-induction-upper-bound}
Let $r$ be a positive integer, and let $\lambda=(\lambda_1,\dots,\lambda_s)$ be an ordered partition of $r$. For every collectable $r$-pattern $P$, consider a positive integer $m_P$. Let $\HH$ be a finite collection of pairwise disjoint subsets $e\su \Z$, each of size $|e|=r$, such that $\HH$ is $\lambda$-partite. Furthermore, assume that for every collectable $r$-pattern $P$, every $P$-clique contained in $\HH$ has size at most $m_P$. Then the size of $\HH$ is bounded by
\begin{equation}
    |\HH|\le 2^{r-s} \sum_{\lambda^{(s)} \succ \lambda^{(s+1)} \succ \dots \succ \lambda^{(r)}}\ \prod_{ P\in \PP(\lambda^{(s)})\cup\dots\cup \PP(\lambda^{(r)}) } m_{P},
\end{equation}
where the sum is over all sequences $\lambda^{(s)} \succ \lambda^{(s+1)} \succ \dots \succ \lambda^{(r)}$ of ordered partitions of $r$ with $\lambda^{(s)}=\lambda$ (note that then for $t=s,\dots,r$, the ordered partition $\lambda^{(t)}$ must automatically have $t$ parts).
\end{theorem}

Note that Theorem~\ref{thm-induction-upper-bound} directly implies Theorem~\ref{thm-paramteters-upper-bound} by taking $\lambda=(r)$ to be the ordered partition of $r$ into a single part. Also recall that Theorem~\ref{thm-paramteters-upper-bound} implies Theorem~\ref{thm-main}.

\begin{proof}[Proof of Theorem~\ref{thm-induction-upper-bound}]
    We prove the theorem by induction on $r-s$ (note that we always have $s\le r$). If $s=r$, i.e.\ if $\lambda=(1,1,\dots,1)$ is the unique ordered partition of $r$ into $r$ parts, then $\HH$ is automatically interval-wise $\lambda$-partite. Thus, by Lemma~\ref{obs-interval-wise} we have
    \[|\HH|\le \prod_{P\in \PP(\lambda)} m_{P}=2^{r-r} \sum_{\lambda^{(s)} \succ \dots \succ \lambda^{(r)}}\  \prod_{P\in \PP(\lambda^{(s)})\cup\dots\cup \PP(\lambda^{(r)})} m_{P},\]
    where the sum has only one summand corresponding to the unique length-1 sequence $\lambda^{(s)} \succ \dots \succ \lambda^{(r)}$ with $\lambda^{(s)}=\lambda^{(r)}=\lambda$.

    So let us now assume that $s<r$ and that we already proved the desired statement for all smaller values of $r-s$.  For every ordered partition $\lambda'$ of $r$ into $s+1$ parts with $\lambda\succ\lambda'$, let us define
    \[M_{\lambda'}=2^{r-s-1} \sum_{\lambda^{(s+1)} \succ \lambda^{(s+2)} \succ \dots \succ \lambda^{(r)}}\ \prod_{P\in \PP(\lambda^{(s+1)})\cup\dots\cup \PP(\lambda^{(r)})} m_{P},\]
    where the sum is over all sequences $\lambda^{(s+1)} \succ \lambda^{(s+2)} \succ \dots \succ \lambda^{(r)}$ of ordered partitions of $r$ with $\lambda^{(s+1)}=\lambda'$. As in Proposition~\ref{prop-two-options}, let $M=\sum_{\lambda'}M_{\lambda'}$ be the sum of these $M_{\lambda'}$ for all ordered partitions $\lambda'$ of $r$ into $s+1$ parts with $\lambda\succ\lambda'$. Then we have
    \[M=\sum_{\lambda'} M_{\lambda'}=2^{r-s-1} \sum_{\lambda^{(s)}\succ \lambda^{(s+1)}\succ \dots \succ \lambda^{(r)}}\ \prod_{ P\in \PP(\lambda^{(s+1)})\cup\dots\cup \PP(\lambda^{(r)}) } m_{P},\]
    where on the right-hand side the sum is over all sequences $\lambda^{(s)} \succ \lambda^{(s+1)} \succ \dots \succ \lambda^{(r)}$ of ordered partitions of $r$ with $\lambda^{(s)}=\lambda$.

    For every ordered partition $\lambda'$ of $r$ into $s+1$ parts with $\lambda\succ\lambda'$, we can apply the inductive assumption for $\lambda'$, and conclude that for every $\lambda'$-partite sub-collection $\HH'\su \HH$, we must have $|\HH'|\le M_{\lambda'}$. Thus, when applying Proposition~\ref{prop-two-options} to $\HH$ and $\lambda$, with the numbers $M_{\lambda'}$ as defined above, option (b) cannot happen. Hence option (a) must happen, and we can find a sub-collection $\HH^*\su \HH$ of size $|\HH^*|\ge |\HH|/(2M)$ such that $\HH^*$ is interval-wise $\lambda$-partite. But now by Lemma~\ref{obs-interval-wise}, we have
    \[|\HH|/(2M)\le |\HH^*|\le \prod_{P\in \PP(\lambda)} m_{P}.\]
    This gives
    \[|\HH|\le 2M\cdot \prod_{P\in \PP(\lambda)} m_{P}=2^{r-s} \sum_{\lambda^{(s)} \succ \lambda^{(s+1)} \succ \dots \succ \lambda^{(r)}}\ \prod_{ P\in \PP(\lambda^{(s)})\cup\dots\cup \PP(\lambda^{(r)}) } m_{P},\]
    as desired, where the sum on the right-hand side is again over all sequences $\lambda^{(s)} \succ \lambda^{(s+1)} \succ \dots \succ \lambda^{(r)}$ of ordered partitions of $r$ with $\lambda^{(s)}=\lambda$.
    \end{proof}

\section{Lower bound Construction}
\label{sect-lower-bound}

In this section, we prove Theorem~\ref{thm-paramteters-lower-bound}, by giving a construction of a large ordered hypergraph matching where all $P$-cliques have at most some given sizes $m_P$. The construction is an appropriate multi-parameter generalization of the argument from \cite[Section 4]{AJKS}. We will actually prove the following stronger form of Theorem~\ref{thm-paramteters-lower-bound}.

\begin{theorem} \label{thm-paramteters-lower-bound-stronger}
Let $r$ be a positive integer, and let $\lambda^{(1)} \succ \lambda^{(2)} \succ \dots \succ \lambda^{(r)}$ be a sequence of ordered partitions of $r$. For every $r$-pattern $P\in \PP(\lambda^{(1)})\cup\dots\cup \PP(\lambda^{(r)})$, consider a positive integer $m_P$. Then there exists an  ordered $r$-uniform hypergraph matching $\HH$ of size
\[|\HH|=\prod_{P\in \PP(\lambda^{(1)})\cup\dots\cup \PP(\lambda^{(r)})} m_{P},\]
such that any two edges $e,f\in \HH$ form an $r$-pattern in $\PP(\lambda^{(1)})\cup\dots\cup \PP(\lambda^{(r)})$, and for each $r$-pattern $P\in \PP(\lambda^{(1)})\cup\dots\cup \PP(\lambda^{(r)})$ every $P$-clique contained in $\HH$ has size at most $m_P$.
\end{theorem}

Note that Theorem~\ref{thm-paramteters-lower-bound-stronger} immediately implies Theorem~\ref{thm-paramteters-lower-bound} by taking $\lambda^{(1)} \succ \lambda^{(2)} \succ \ldots \succ \lambda^{(r)}$ to be a sequence of ordered partitions of $r$ achieving the maximum in (\ref{eq_main_lower}).

Following \cite{AJKS}, for an ordered $r$-uniform hypergraph matching $\HH$ and an $r$-pattern $P$, let $L_P(\HH)$ denote the size of the largest $P$-clique contained in $\HH$. Then the conditions on $\HH$ in Theorem~\ref{thm-paramteters-lower-bound-stronger} can be rephrased as $L_P(\HH)\le m_P$ for all $r$-patterns $P\in \PP(\lambda^{(1)})\cup\dots\cup \PP(\lambda^{(r)})$ and $L_P(\HH)=1$ for all other $r$-patterns $P\not\in \PP(\lambda^{(1)})\cup\dots\cup \PP(\lambda^{(r)})$.

An ordered $r$-uniform hypergraph matching $\HH$ is \emph{$r$-partite}, if in the underlying ordering of the vertex set is such that all the first vertices of all the edges come before all the second vertices of all the edges, which come before all the third vertices of all the edges, and so on. More formally, we can say that $\HH$ is $r$-partite, if as an ordered $r$-uniform hypergraph matching $\HH$ is isomorphic to a finite collection of pairwise disjoint subsets $e\su \Z$, each of size $|e|=r$, which is $\lambda^{(r)}$-partite according to Definition~\ref{def-s-partite} for the unique ordered partition $\lambda^{(r)}=(1,\dots,1)$ of $r$ into $r$ parts.

\begin{figure}
    \centering
    \begin{tikzpicture}[x=0.6cm,y=0.6cm]
    \filldraw[black] (1,3) circle (3pt);
    \filldraw[black] (5,3) circle (3pt);
    \filldraw[black] (9,3) circle (3pt);
    \filldraw[black!60] (2,2) circle (3pt);
    \filldraw[black!60] (6,2) circle (3pt);
    \filldraw[black!60] (7,2) circle (3pt);
    \filldraw[gray!30] (3,1) circle (3pt);
    \filldraw[gray!30] (4,1) circle (3pt);
    \filldraw[gray!30] (8,1) circle (3pt);
    \draw[dotted] (0.5,0.5) -- (0.5,3.5);
    \draw[dotted] (3.5,0.5) -- (3.5,3.5);
    \draw[dotted] (6.5,0.5) -- (6.5,3.5);
    \draw[dotted] (9.5,0.5) -- (9.5,3.5);
    \end{tikzpicture}
    \caption{The hypergraph matching $\HH_3 = \{ \{1,5,9\}, \{2,6,7\},\{3, 4,8\} \}$ is $r$-partite for $r=3$. Again, the rows in the picture correspond to the edges in $\HH_3$.}
    \label{fig:figure-example-r-partite}
\end{figure}

For example, when taking $r=3$, then the hypergraph matching
\[
\HH_3 = \{ \{1,5,9\}, \{2,6,7\} ,\{3, 4,8\}\}
\]
with the natural vertex ordering is $r$-partite (see Figure~\ref{fig:figure-example-r-partite}).

For our construction proving Theorem~\ref{thm-paramteters-lower-bound-stronger}, we will use a notion of a {\em blow-up} introduced in \cite{DGR22} and later also used in \cite{AJKS}. This notion is defined as follows, see also Figure~\ref{fig:blow-up} for an illustration.

\begin{definition}\label{def-blow-up}
    Let $\HH$ and $\HH_r$ be ordered $r$-uniform hypergraph matchings such that $\HH_r$ is $r$-partite. Then we define the blow-up $\HH[\HH_r]$ to be the ordered $r$-uniform hypergraph matching of size $|\HH[\HH_r]|=|\HH|\cdot |\HH_r|$ defined as follows. For the vertex set $\HH[\HH_r]$, starting with $\HH$, let us replace every vertex of $\HH$ with $|\HH_r|$ copies of itself (right after each other as a contiguous interval in the ordering of the vertex set). Then, for every edge $e\in \HH$, let us take a matching of $|\HH_r|$ edges between these copies of the vertices of $e$, in the shape prescribed by $\HH_r$.

    More formally, this can be described when modelling $\HH$ and $\HH_r$ by pairwise disjoint size-$r$ subsets of $\Z$. As $\HH_r$ is $r$-partite, there exists a sequence $x_0 < x_1 < \ldots < x_r$ of real numbers such that $|e\cap (x_{i-1},x_i)|=1$ for all $i=1,\dots,r$ and all $e\in \HH_r$. Take a large integer $N$ such that $N>x_i-x_{i-1}$ for $i=1,\dots,r$, and define the blow-up $\HH[\HH_r]$ to be the collection of size-$r$ subsets of $\Z$ of the form
    \[\{ N x_1 + y_1, N x_2 + y_2, \ldots, N x_r + y_r \}\]
    for all edges $\{x_1,\dots,x_r\}\in \HH$ with $x_1<\dots<x_r$ and all edges $\{y_1,\dots,y_r\}\in \HH_r$ with $y_1<\dots<y_r$.
\end{definition}

\begin{figure}
    \centering
    \begin{tikzpicture}[x=0.6cm,y=0.6cm]
    \filldraw[red] (1,9) circle (3pt);
    \filldraw[red] (11,9) circle (3pt);
    \filldraw[red] (21,9) circle (3pt);
    \filldraw[red!60] (2,8) circle (3pt);
    \filldraw[red!60] (12,8) circle (3pt);
    \filldraw[red!60] (19,8) circle (3pt);
    \filldraw[red!30] (3,7) circle (3pt);
    \filldraw[red!30] (10,7) circle (3pt);
    \filldraw[red!30] (20,7) circle (3pt);
    \filldraw[blue] (4,6) circle (3pt);
    \filldraw[blue] (14,6) circle (3pt);
    \filldraw[blue] (18,6) circle (3pt);
    \filldraw[blue!60] (5,5) circle (3pt);
    \filldraw[blue!60] (15,5) circle (3pt);
    \filldraw[blue!60] (16,5) circle (3pt);
    \filldraw[blue!30] (6,4) circle (3pt);
    \filldraw[blue!30] (13,4) circle (3pt);
    \filldraw[blue!30] (17,4) circle (3pt);
    \filldraw[Green] (7,3) circle (3pt);
    \filldraw[Green] (23,3) circle (3pt);
    \filldraw[Green] (27,3) circle (3pt);
    \filldraw[Green!60] (8,2) circle (3pt);
    \filldraw[Green!60] (24,2) circle (3pt);
    \filldraw[Green!60] (25,2) circle (3pt);
    \filldraw[Green!30] (9,1) circle (3pt);
    \filldraw[Green!30] (22,1) circle (3pt);
    \filldraw[Green!30] (26,1) circle (3pt);
    \end{tikzpicture}
    \caption{Example of a blow-up $\HH[\HH_3]$ with $\HH = \{ \{1,4,7\}, \{2, 5,6\}, \{3,8,9\} \}$ as on the left side of Figure~\ref{fig:figure-examples-partite}, and $\HH_3 = \{ \{1,5,9\}, \{2,6,7\},\{3, 4,8\} \}$ as in Figure~\ref{fig:figure-example-r-partite}. A copy of $\HH_3$ is placed inside each edge of $\HH$.}
    \label{fig:blow-up}
\end{figure}


 Note that up to an isomorphism of ordered $r$-uniform hypergraph matchings, the second description in the definition above does not depend on the choice of $N$ or the choices made when modelling $\HH$ and $\HH_r$ as collections of subsets of $\Z$.

For ordered $r$-uniform hypergraph matchings $\HH$ and $\HH_r$ as in Definition~\ref{def-blow-up}, we have
\begin{equation}\label{Lp}
L_P(\HH[\HH_r]) \le L_P(\HH) L_P(\HH_r)
\end{equation}
for every $r$-pattern $P$. Indeed, for any $P$-clique in $\HH[\HH_r]$, the corresponding edges in $\HH$ must also form a $P$-clique (indeed, any two edges in $\HH[\HH_r]$ coming from different edges in $\HH$ form the same $r$-pattern as the corresponding two edges in $\HH$). Furthermore, for every edge $e\in \HH$, the size of the largest $P$-clique among the edges in $\HH[\HH_r]$ corresponding to $e$ is bounded by $L_P(\HH_r)$ (indeed, such a $P$-clique in $\HH[\HH_r]$ corresponds to a $P$-clique in $\HH_r$).

The other ingredient in our proof Theorem~\ref{thm-paramteters-lower-bound-stronger} is the following statement, giving a construction of an $r$-partite ordered $r$-uniform hypergraph matching $\HH$ with prescribed bounds for the clique sizes $L_P(\HH)$. This statement is a direct consequence of the second part of Theorem`\ref{thm_es}.

\begin{lemma}\label{obs-construction-partite}
    Let $r$ be a positive integer, and let $\lambda^{(r)}=(1,\dots,1)$ be the unique ordered partition of $r$ into $r$ parts. For every $r$-pattern $P\in \PP(\lambda^{(r)})$, consider a positive integer $m_P$. Then there exists an $r$-partite ordered $r$-uniform hypergraph matching $\HH$ of size
\[|\HH|=\prod_{P\in \PP(\lambda^{(r)})} m_{P},\]
such that $L_P(\HH)\le m_P$ for all $r$-patterns $P\in \PP(\lambda^{(r)})$ and $L_P(\HH)=1$ for all other $r$-patterns $P\not\in \PP(\lambda^{(r)})$.
\end{lemma}

We note that for any $r$-partite ordered $r$-uniform hypergraph matching $\HH$ we automatically have $L_{P}(\HH)=1$ for any $P\not\in \PP(\lambda^{(r)})$. Indeed, for any two edges $e,f\in \HH$, the first vertices of $e$ and $f$ are both before the second vertices of $e$ and $f$, which are both before the third vertices of $e$ and $f$ and so on. Hence the $r$-pattern formed by $e$ and $f$ consists of $r$ blocks that each have one $\rm A$ and one $\rm B$, so the $r$-pattern has a block partition into $r$ parts and therefore belongs to $\PP(\lambda^{(r)})$.

\begin{proof}[Proof of Lemma~\ref{obs-construction-partite}]
    Recall from the discussion at the end of Section~\ref{sect-preliminaries} that any $r$-pattern $P\in \PP(\lambda^{(r)})$ is of the form $P=P(\lambda^{(r)},\tau)$ for a unique function $\tau:\{1,\dots,r\}\to \{1,-1\}$ with $\tau(1)=1$. For every such function $\tau$, let us define $m_\tau=m_{P(\lambda^{(r)},\tau)}$. By the second part of Theorem~\ref{thm_es}, there exists a collection of points $Z\su \mathbb{R}^r$ of size $|Z|=\prod_{\tau} m_{\tau}=\prod_{P\in \PP(\lambda^{(r)})} m_{P}$ such that the $|Z|\cdot r$ coordinates of all the points in $Z$ are all distinct, and such that for every function $\tau:\{1,\dots,r\}\to \{1,-1\}$ with $\tau(1)=1$, every $\tau$-monotone sequence in $Z$ has length at most $m_\tau=m_{P(\lambda^{(r)},\tau)}$.
    
    Without loss of generality, we may assume that all points in $Z$ have positive coordinates (otherwise, we can translate the set $Z$ accordingly). Let $N>0$ be a number larger than all coordinates of all points in $Z$. Now, for every point $z=(z_1,\dots,z_r)\in Z\su \mathbb{R}^r$, let us define $e(z)=\{N+ z_1, 2N + z_2, \ldots, r N+z_r \}\su \mathbb{R}$, and note that the sets $e(z)$ for $z\in Z$ are pairwise disjoint. Taking $\HH=\{e(z)\mid z\in Z\}$ to be the collection of the sets $e(z)$ for all $z\in Z$, we obtain an ordered $r$-uniform hypergraph matching $\HH$ by considering the natural ordering on $\mathbb{R}$ on the vertex set. Note that $\HH$ has size $|\HH|=|Z|=\prod_{P\in \PP(\lambda^{(r)})} m_{P}$. Furthermore, $\HH$ is $r$-partite (by the choice of $N$) and so in particular we have $L_{P}(\HH)=1$ for any $P\not\in \PP(\lambda^{(r)})$. It remains to check $L_{P}(\HH)\le m_P$ for every $P\in \PP(\lambda^{(r)})$.

    So let $P\in \PP(\lambda^{(r)})$ and let $\tau:\{1,\dots,r\}\to \{1,-1\}$ with $\tau(1)=1$ be such that $P=P(\lambda^{(r)},\tau)$, then we have $m_\tau=m_P$. Let $z^{(1)},\dots,z^{(t)}\in Z$ be such that the edges $e(z^{(1)}),\dots,e(z^{(t)})$ form $P$-clique in $\HH$. Then for any $1\le \ell<\ell'\le t$, for any $i\in \{1,\dots,r\}$ with $\tau(i)=1$, the $i$-th vertex $iN+z^{(\ell)}_i$ of $e(z^{(\ell)})$ comes before the $i$-th vertex $iN+z^{(\ell')}_i$ of $e(z^{(\ell')})$ and so we have $z^{(\ell)}_i<z^{(\ell')}_i$. Similarly for any $i\in \{1,\dots,r\}$ with $\tau(i)=1$, the $i$-th vertex of $e(z^{(\ell)})$ comes after the $i$-th vertex of $e(z^{(\ell')})$ and so we have $z^{(\ell)}_i>z^{(\ell')}_i$. Thus, for every $i\in \{1,\dots,r\}$ with $\tau(i)=1$, the $i$-th coordinates of the points $z^{(1)},\dots,z^{(t)}$ form an increasing sequence 
    $z^{(1)}_1<\dots<z^{(t)}_i$, and for every $i\in \{1,\dots,r\}$ with $\tau(i)=-1$, they form a decreasing sequence $z^{(1)}_1>\dots>z^{(t)}_i$. Thus, $z^{(1)},\dots,z^{(t)}$ is a $\tau$-monotone sequence of points in $Z$ and hence $t\le m_\tau\le m_P$. This shows that $L_P(\HH)\le m_P$, as desired.
\end{proof}

Finally, we are ready to prove Theorem~\ref{thm-paramteters-lower-bound-stronger} (implying Theorem~\ref{thm-paramteters-upper-bound}).

\begin{proof}[Proof of Theorem~\ref{thm-paramteters-lower-bound-stronger}]
We prove the theorem by induction on $r$. Note that the base case $r=1$ is trivial, so let us assume $r\ge 2$ and that we already proved the theorem for $r-1$.

In the sequence $\lambda^{(1)} \succ \lambda^{(2)} \succ \dots \succ \lambda^{(r)}$, each ordered partition $\lambda^{(s)}$ for $s=1,\dots,r$ must have precisely $s$ parts. In particular, the partition $\lambda^{(r-1)}$ has the form $\lambda^{(r-1)}=(1,\dots,1,2,1,\dots,1)$. Let $j\in \{1,\dots,r-1\}$ be the unique index such that $\lambda^{(r-1)}_j=2$. If we imagine $\lambda^{(1)} \succ \lambda^{(2)} \succ \dots \succ \lambda^{(r)}$ as partitions of the set $\{1,\dots,r\}$ into intervals, then the partition  $\lambda^{(r-1)}$ consists of the set $\{j,j+1\}$ and singleton sets for all the remaining elements. Hence $j$ and $j+1$ are in the same set for each of the partitions $\lambda^{(1)} \succ \lambda^{(2)} \succ \dots \succ \lambda^{(r-1)}$. If we ignore $j$ and relabel the remaining elements accordingly (i.e.\ $j+1$ is relabeled $j$, $j+2$ is relabeled $j+1$, and so on), $\lambda^{(1)} \succ \lambda^{(2)} \succ \dots \succ \lambda^{(r-1)}$ turns into a sequence of ordered partitions $\mu^{(1)} \succ \mu^{(2)} \succ \dots \succ \mu^{(r-1)}$ of $r-1$. For $s=1,\dots,r-1$ we obtain $\lambda^{(j)}$ from $\mu^{(j)}$ by doubling the element $j$ (keeping both copies in the same set of the partition), and relabeling the ground set into $\{1,\dots,r\}$ accordingly.

Now, for $s=1,\dots,r-1$, every $r$-pattern $P\in \PP(\lambda^{(s)})$ turns into an $(r-1)$-pattern $P'\in \PP(\mu^{(s)})$ when omitting the $j$-th $\rm A$ and the $j$-th $\rm B$ in $P$ (note that these occur in the same block, i.e.\ right adjacent to the $(j+1)$-th $\rm A$ and the $(j+1)$-th $\rm B$, respectively). Conversely, $P$ can be recovered from $P'\in \PP(\mu^{(s)})$ by doubling the $j$-th $\rm A$
in $P$ into $\rm{AA}$ and doubling the $j$-th $\rm B$
in $P$ into $\rm{BB}$. Thus, these operations give natural bijections between $\PP(\lambda^{(s)})$ and $\PP(\mu^{(s)})$.

For every $(r-1)$-pattern $P'\in \PP(\mu^{(1)})\cup\dots\cup \PP(\mu^{(r-1)})$, let us define $m_{P'}=m_P$, where $P\in \PP(\lambda^{(1)})\cup\dots\cup \PP(\lambda^{(r-1)})$ is the $r$-pattern obtained from $P'$ by doubling the $j$-th $\rm A$ and the $j$-th $\rm B$. Then, applying the induction hypothesis to the sequence $\mu^{(1)} \succ \mu^{(2)} \succ \dots \succ \mu^{(r-1)}$ of ordered partitions of $r-1$, we obtain an ordered $(r-1)$-uniform hypergraph matching $\HH'$ of size
\[|\HH'|=\prod_{P'\in \PP(\mu^{(1)})\cup\dots\cup \PP(\mu^{(r-1)})} m_{P'}=\prod_{P\in \PP(\lambda^{(1)})\cup\dots\cup \PP(\lambda^{(r-1)})} m_{P},\]
such that $L_{P'}(\HH')\le m_{P'}$ for all $(r-1)$-patterns $P'\in \PP(\mu^{(1)})\cup\dots\cup \PP(\mu^{(r-1)})$ and $L_{P'}(\HH')=1$ for all other $(r-1)$-patterns $P'\not\in \PP(\mu^{(1)})\cup\dots\cup \PP(\mu^{(r-1)})$.

From $\HH'$, we can now obtain an ordered $r$-uniform hypergraph matching $\HH^*$ with $|\HH^*|=|\HH'|$ by doubling the $j$-th element of every edge $e'\in \HH'$. More formally, for every edge $e'\in \HH'$, we add a new element to $e'$ right before the $j$-the element of $e'$ (in the given ordering of the underlying vertex set of $\HH'$). For the resulting edge $e\in \HH^*$, we can recover the original edge $e'\in \HH'$ by deleting the $j$-the element of $e$. If $e,f\in \HH^*$ form some $r$-pattern $P$, then the $(r-1)$-pattern formed by the corresponding edges $e',f'\in \HH'$ is the $(r-1)$-pattern $P'$ obtained from $P$ by omitting the $j$-th $\rm A$ and the $j$-th $\rm B$. Thus, by the properties of $\HH'$, we have $L_P(\HH)\le m_P$ for all $r$-patterns $P\in \PP(\lambda^{(1)})\cup\dots\cup \PP(\lambda^{(r-1)})$ and $L_P(\HH)=1$ for all other $r$-patterns $P\not\in \PP(\lambda^{(1)})\cup\dots\cup \PP(\lambda^{(r-1)})$.

Finally, recall that $\lambda^{(r)}=(1,\dots,1)$ is the unique ordered partition of $r$ into $r$ parts. Let $\HH_r$ be an $r$-partite ordered $r$-uniform hypergraph matching as in Lemma~\ref{obs-construction-partite}. Now, the blow-up $\HH=\HH^*[\HH_r]$ has size
\[|\HH|=|\HH^*|\cdot |\HH_r|=|\HH'|\cdot |\HH_r|=\prod_{P\in \PP(\lambda^{(1)})\cup\dots\cup \PP(\lambda^{(r-1)})} m_{P}\cdot \prod_{P\in \PP(\lambda^{(r)})} m_{P}=\prod_{P\in \PP(\lambda^{(1)})\cup\dots\cup \PP(\lambda^{(r)})} m_{P},\]
and by (\ref{Lp}) for every $r$-pattern $P$ we have $L_P(\HH)\le L_P(\HH^*)\cdot L_P(\HH_r)$. Hence $L_P(\HH)\le m_P\cdot 1=m_P$ for all $P\in \PP(\lambda^{(1)})\cup\dots\cup \PP(\lambda^{(r-1)})$, and $L_P(\HH)\le 1\cdot m_P=m_P$ for all $P\in \PP(\lambda^{(r)})$, and also $L_P(\HH)\le 1\cdot 1=1$ for all $P\not\in \PP(\lambda^{(1)})\cup\dots\cup \PP(\lambda^{(r)})$. Thus, we have constructed the desired ordered $r$-uniform hypergraph matching $\HH$.
\end{proof}

\bibliographystyle{amsplain}

\begin{aicauthors}
\begin{authorinfo}[pgom]
  Lisa Sauermann\\
  Institute for Applied Mathematics, University of Bonn\\
  Bonn, Germany\\
  sauermann\imageat{}iam\imagedot{}uni-bonn\imagedot{}de \\
  \url{https://www.iam.uni-bonn.de/users/sauermann/home}
\end{authorinfo}
\begin{authorinfo}[johan]
  Dmitrii Zakharov\\
  Massachusetts Institute of Technology\\
  Cambridge, Massachusetts, United States of America\\
  zakhdm\imageat{}mit\imagedot{}edu \\
  \url{https://math.mit.edu/directory/profile.html?pid=2452}
\end{authorinfo}

\end{aicauthors}

\end{document}